\documentclass[12pt,a4paper]{amsart}

\usepackage{amsfonts, amsmath, amssymb, amsthm, amscd}
\usepackage{anysize}
\usepackage{graphicx}
\usepackage{subcaption}
\usepackage{hyperref}
\usepackage{cool}
\usepackage{color}

\marginsize{1.6cm}{1.6cm}{1.6cm}{1.6cm}
\usepackage[mathlines]{lineno}

\usepackage{etoolbox} 

\newcommand*\linenomathpatchAMS[1]{%
  \expandafter\pretocmd\csname #1\endcsname {\linenomathAMS}{}{}%
  \expandafter\pretocmd\csname #1*\endcsname{\linenomathAMS}{}{}%
  \expandafter\apptocmd\csname end#1\endcsname {\endlinenomath}{}{}%
  \expandafter\apptocmd\csname end#1*\endcsname{\endlinenomath}{}{}%
}

\expandafter\ifx\linenomath\linenomathWithnumbers
  \let\linenomathAMS\linenomathWithnumbers
  \patchcmd\linenomathAMS{\advance\postdisplaypenalty\linenopenalty}{}{}{}
\else
  \let\linenomathAMS\linenomathNonumbers
\fi

\linenomathpatchAMS{gather}
\linenomathpatchAMS{multline}
\linenomathpatchAMS{align}
\linenomathpatchAMS{alignat}
\linenomathpatchAMS{flalign}

\newtheorem{theorem}{Theorem}
\newtheorem{lemma}[theorem]{Lemma}
\newtheorem{conjecture}[theorem]{Conjecture}

\theoremstyle{definition}

\newcommand\BB{\mathbb B}
\newcommand\DD{\mathbb D}
\newcommand\EE{\mathbb E}
\newcommand\RR{\mathbb R}
\newcommand\SSS{\mathbb S}
\newcommand\Sphere{\mathbb{S}^2}
\newcommand{\Xsym}          {{X}^{\rm sym}}
\newcommand{\model}         {{\mathcal M}}
\newcommand{\NDef}[1]       {{\Delta}{{#1}}}
\newcommand{\Grass}[2]      {{\mathbf{G}}{({#1},{#2})}}
\newcommand{\Linear}        {{\mathcal{T}}}

\newcommand{\diff}{\,\mathrm{d}}
\newcommand{\One}{\mathbf{1}}
\newcounter{fig}

\newcommand{\OneEmpty}[1]   {{{\mathbf{1}_{\rm empty}{[{#1}]}}}}
\newcommand{\OneFacet}[1]   {{{\mathbf{1}_{\rm facet}{[{#1}]}}}}
\newcommand{\Expected}[1]   {{{\mathbf E}{[{#1}]}}}
\newcommand{\PP}            {{{\mathbf P}}}
\newcommand{\Probable}[1]   {{{\PP}{[{#1}]}}}
\newcommand{\Bcc}[1]        {{\mathrm{Bg}{({#1})}}}
\newcommand{\Scc}[1]        {{\mathrm{Sm}{({#1})}}}
\newcommand{\Edist}[2]      {{\|{#1}-{#2}\|}}
\newcommand{\scalprod}[2]   {{\langle #1 , #2 \rangle}}
\newcommand{\Width}         {{{\rm W}}}
\newcommand{\Area}          {{{\rm A}}}
\newcommand{\Vol}           {{{\rm V}}}
\newcommand{\Length}        {{{\rm L}}}
\newcommand{\density}       {{\varrho}}
\newcommand{\conv}[1]       {{\rm conv\,}{#1}}
\newcommand{\proj}[2]       {{\rm proj}_{#1}{({#2})}}
\newcommand{\BetaF}[2]      {{{\rm B}}{({#1},{#2})}}
\newcommand{\BesselF}[1]    {{\,{\rm I}_{#1}}}

\newcommand{\Skip}[1]       {}

\title{The Beauty of Random Polytopes Inscribed in the 2-sphere}

\author{Arseniy~Akopyan}
\author{Herbert Edelsbrunner}
\author{Anton Nikitenko}

\address{Institute of Science and Technology Austria (IST Austria), Am Campus~1, 3400 Klosterneuburg, Austria}
\email{akopjan@gmail.com, edels@ist.ac.at, anton.nikitenko@ist.ac.at}

\begin{document}

\keywords{Blaschke--Petkantschin, random inscribed polytopes, Archimedes' Lemma, intrinsic volumes,
  mean width, area, volume, edge length, distance, Poisson point process,
	homeoid density, Crofton's formula}

\begin{abstract}
  Consider a random set of points on the unit sphere in $\RR^d$,
  which can be either uniformly sampled or a Poisson point process.
Its convex hull is a random inscribed polytope, whose boundary approximates the sphere.
We focus on the case $d=3$, for which there are elementary proofs
and fascinating formulas for metric properties.
In particular, we study the fraction of acute facets, the expected intrinsic volumes,
the total edge length, and the distance to a fixed point.
Finally we generalize the results to the ellipsoid with homeoid density.
\end{abstract}
\maketitle

\section{Introduction}
\label{sec:1-Introduction}

The study of random geometric structures has been an active field
of mathematics for the last several decades.
With an effort of being as general as possible,
results often end up as cumbersome formulas with multiple parameters,
sometimes being recurrent, and involving special functions.
As such, many beautiful formulas remained hidden,
despite being special cases of more general ones.
For example, the \emph{expected intrinsic volumes} of a random polytope have been
computed first in \cite{buchta1985stochastical} for \emph{spherical polytopes},
and later in \cite{Kab17} for \emph{Beta polytopes},
but the following exciting expressions for the $2$-sphere
were overlooked in the first and lost within
a large number of corollaries in the second article:
\begin{align}
  \Expected{\Width(X_n)}  &=    \Width(\BB^3) \cdot \tfrac{n-1}{n+1},
    \label{eq:intro:width}   \\
  \Expected{\Area(X_n)}   &=\;  \Area(\BB^3) \cdot \tfrac{n-1}{n+1} \tfrac{n-2}{n+2},
    \label{eq:intro:area}   \\
  \Expected{\Vol(X_n)}    &=\;  \Vol(\BB^3) \cdot \tfrac{n-1}{n+1} \tfrac{n-2}{n+2}
                                                                   \tfrac{n-3}{n+3},
    \label{eq:intro:volume}
\end{align}
in which $\Width, \Area, \Vol$ map a $3$-dimensional convex body to its mean width,
surface area, and volume;
and $X_n$ is the convex hull of $n$ points chosen uniformly at random on the $2$-sphere.
We prove a similar relation for the total edge length and extend
\eqref{eq:intro:width}, \eqref{eq:intro:area}, \eqref{eq:intro:volume}
to random centrally symmetric polytopes.
In addition, we derive the rather similar corresponding relations for
a stationary Poisson point process:
\begin{align}
  \Expected{\Width(X_\density)} &= \Width(\BB^3)
    \cdot 2 \pi \density^{0.5} e^{-2 \pi \density} {\BesselF{1.5} (2\pi\density)},
    \label{eqn:Intro-EIV-II1} \\
  \Expected{\Area(X_\density)}  &= \; \Area(\BB^3)
    \cdot 2 \pi \density^{0.5} e^{-2 \pi \density} {\BesselF{2.5} (2\pi\density)},
    \label{eqn:Intro-EIV-II2} \\
  \Expected{\Vol(X_\density)}   &= \; \Vol(\BB^3)
    \cdot 2 \pi \density^{0.5} e^{-2 \pi \density} {\BesselF{3.5} (2\pi\density)},
    \label{eqn:Intro-EIV-II3}
\end{align}
in which $\BesselF{\alpha} (x)$ is the modified Bessel function of the first kind.
The generic proofs tend to be probabilistically analytic,
hiding the beautiful geometry implied by the formulas.
An example is the \emph{Blaschke--Petkantschin type} formula for the sphere \cite{EN18},
which is sufficiently powerful to compute expectations of metric properties
of random inscribed polytopes, but the authors
overlooked its simple interpretation,
namely that for a random $p$-simplex inscribed in the $n$-sphere,
its shape and its size are independent.
A similar statement holds in Euclidean space, but this is beyond the scope
of this paper.

\medskip
All of this inspired us to study the special case of random polytopes inscribed
in the $2$-sphere,
with the aim of casting light on the geometric intuition
that works behind the scenes.
By minimizing the use of heavy machinery, we get intuitive geometric proofs
that appeal to our sense of mathematical beauty.
The results we present --- some known and some new --- tend to have inspiringly
simple form, even if we miss the deeper symmetries that govern them.

\subsection*{Outline}
{Section \ref{sec:2-Experiments_and_Motivation} motivates the study of
random inscribed polytopes with result of computational experiments
that give evidence of a strong correlation between their intrinsic volumes.
Section \ref{sec:3-Archimedes} collects geometric facts Archimedes
would have established three centuries BC if probability would have been
a subject of inquiry back then.
Section \ref{sec:4-ShapeandSize} recalls the \emph{independence of shape and size}
and uses it to prove that a random triangle bounding a random polytope inscribed
in the $2$-sphere is acute with probability $\tfrac{1}{2}$.
Section \ref{sec:5-IntrinsicVolume} uses a geometric approach to
compute the expected \emph{intrinsic volumes} of a random inscribed polytope.
We do this for the uniform distribution, for which we also consider
centrally symmetric polytopes, and for Poisson point processes.
Section \ref{sec:6-LengthandDistance} studies the \emph{total edge length}
of a random inscribed polytope---for which it proves a formula that is
surprisingly similar to \eqref{eqn:Intro-EIV-II1}
to \eqref{eqn:Intro-EIV-II3}---as well as the \emph{minimum distance}
of the vertices to a fixed point on the $2$-sphere.
providing evidence for strong correlation between the intrinsic volumes.
Section \ref{sec:7-Deficiencies} probes how far random inscribed polytopes
are from maximizing the intrinsic volumes.
Section \ref{sec:8-Ellipsoid} discusses an application to the distribution
of electrons on an ellipsoid.
Section \ref{sec:9-Discussion} concludes the paper.

\bigskip
\section{{\sc Experiments and Motivation}}
\label{sec:2-Experiments_and_Motivation}

What if we could tell all intrinsic volumes of a polytope knowing just one of them?
The experiments show that the triplets of volumes concentrate along a curve,
as we now explain.
In the subsequent sections, we will show where these curves originate from.
To begin, we show the distributions of the intrinsic volumes
of randomly generated inscribed polytopes in Figure \ref{fig:distributions}.
Considering the mean width, area, and volume, in this sequence,
we see that the normalized expectations get progressively smaller,
and the distributions get progressively wider.
\begin{figure}[ht]%
  \centering
  \includegraphics[scale=0.7]{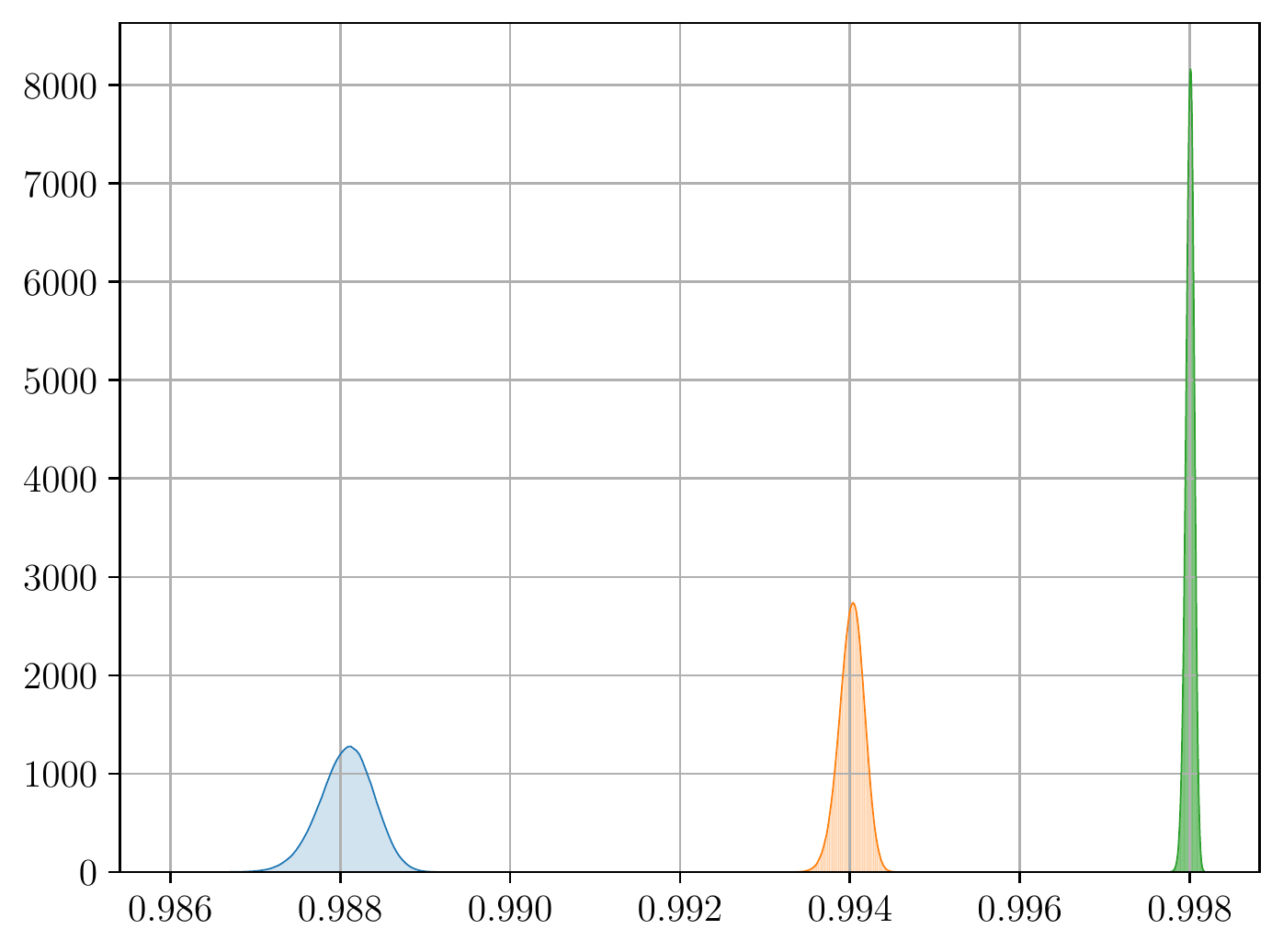}
  \caption{\footnotesize{From \emph{left} to \emph{right}: the distributions
    of the normalized mean width, area, volume of a random inscribed polytope
    with $n = 1000$ vertices in $m = 1\,000\,000$ experiments.}}
  \label{fig:distributions}%
\end{figure}
To further visualize these results, consider
the curve $\gamma \colon [3, \infty) \to \RR$ defined by
\begin{align}
  \gamma(t) &= \left( \Width(\BB^3) \, \tfrac{t-1}{t+1};
                      ~\Area(\BB^3) \, \tfrac{t-1}{t+1} \tfrac{t-2}{t+2};
		      ~\Vol(\BB^3)  \,
                         \tfrac{t-1}{t+1} \tfrac{t-2}{t+2} \tfrac{t-3}{t+3}\right),
\end{align}
and note that it maps positive integers $t = n$ to the triplets of
expected intrinsic volumes;
compare with \eqref{eq:intro:width}, \eqref{eq:intro:area}, \eqref{eq:intro:volume}.
Dropping the intrinsic volumes of the ball, we get the three normalized
expectations, which we note decrease from left to right;
compare with Figure \ref{fig:distributions}.
These inequalities generalize to the normalized intrinsic volumes of any
inscribed polytope:
\begin{align}
  \frac{\Width(X_n)}{\Width(\BB^3)}
    \geq \frac{\Area(X_n)}{\Area(\BB^3)}
    \geq \frac{\Vol(X_n)}{\Vol(\BB^3)} ,
\end{align}
no matter whether $X_n$ is chosen randomly or constructed.
The inequality between the area and the volume follows from the easy observation
that the height of every tetrahedron connecting a triangular facet to the origin
has height less than $1$.
The same argument together with the Crofton formula
applied to the planar projections proves the inequality
between the mean width and the area.

Our experiments show that the three intrinsic volumes deviate from the
expected values in a highly correlated manner.
Indeed, in Figure \ref{fig:moment-curve} we see how the intrinsic volumes
hug the graph of $\gamma$ even when they are far from the expected values.
\begin{figure}[ht]%
  \centering
  \begin{subfigure}[t]{0.45\textwidth}
    \includegraphics[scale=0.50]{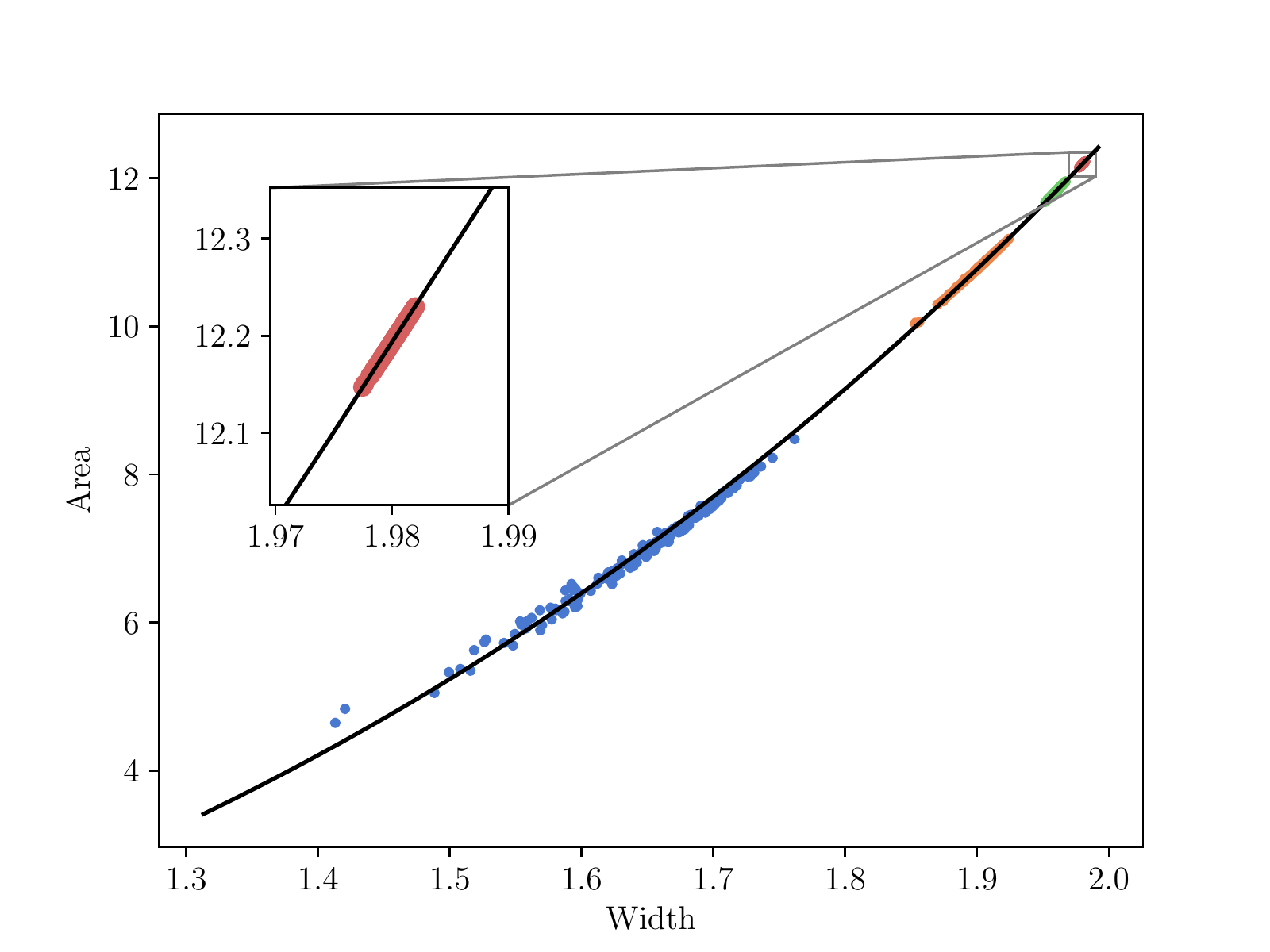}
  \end{subfigure}
  \begin{subfigure}[t]{0.45\textwidth}
    \includegraphics[scale=0.50]{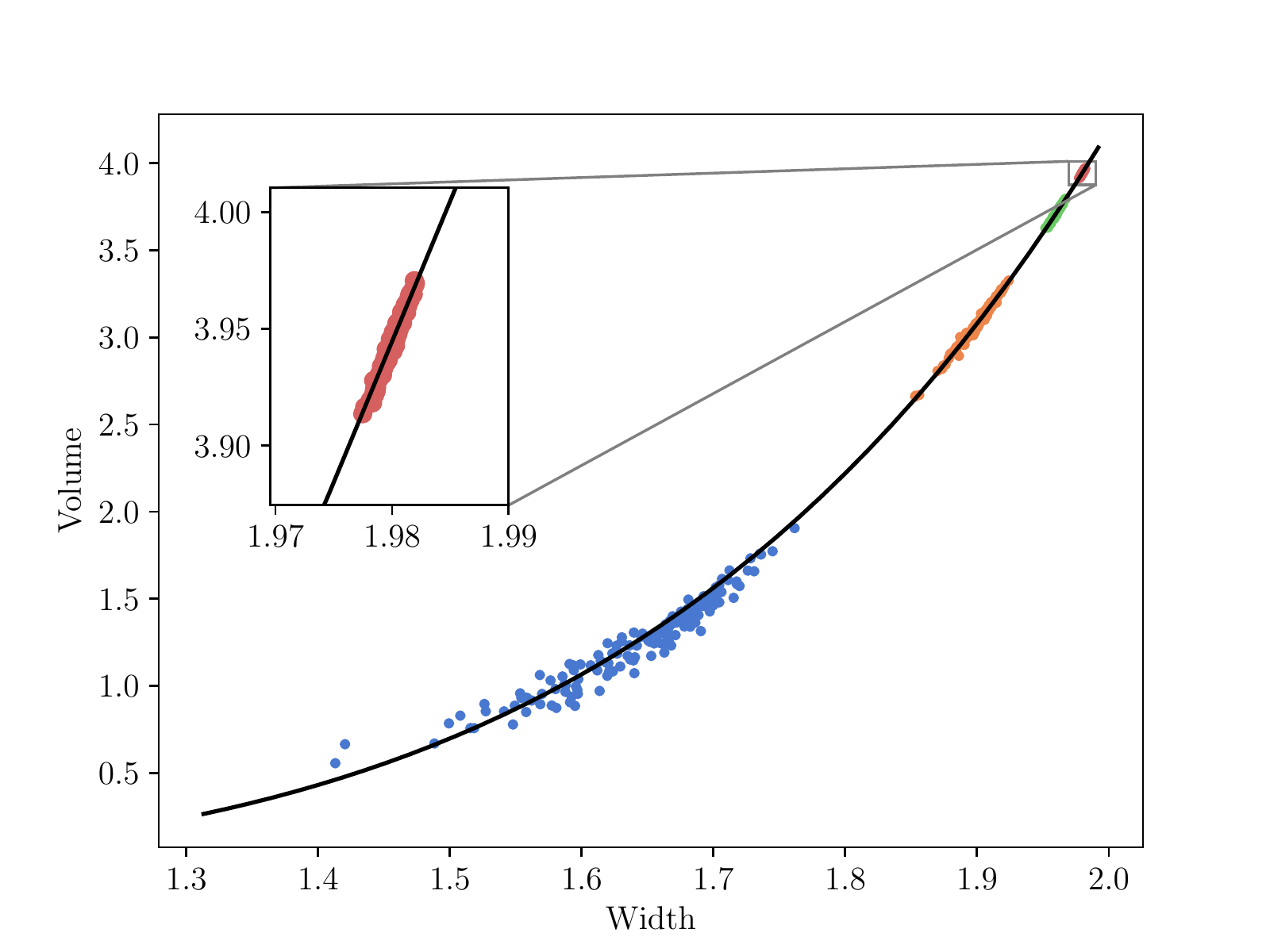}
  \end{subfigure}
  \caption{\footnotesize{Projections of the graph of $\gamma$ and the triplets of
    expected intrinsic volumes into the width-area plane on the \emph{left}
    and the width-volume plane on the \emph{right}.
    \emph{Top:} the $150$ \emph{blue}, \emph{orange}, \emph{green}, and \emph{red}
    points belong to polytopes with $10$, $40$, $100$, and $200$ vertices each.}}
  \label{fig:moment-curve}
\end{figure}
In the two panels, we see four families of random polytopes with
$10$, $40$, $100$, and $200$ vertices, respectively.
As shown in the inserts,
the surprisingly tight fit to the curve can even be observed for random
polytopes with $n = 200$ points, for which the difference
between minimum and maximum intrinsic volume is on the order of $10^{-2}$.
Given one intrinsic volume of a randomly generated polytope,
we can therefore reasonably well predict the other two.
For example, given the mean width, $w$, we can invert \eqref{eq:intro:width}
to get $n(w) = \frac{2+w}{2-w}$,
and plugging $n=n(w)$ into \eqref{eq:intro:area} and \eqref{eq:intro:volume},
we get $A(w) = 4 \pi w \frac{3w-2}{6-w}$
and $V(w) = \frac{4 \pi}{3} w \frac{3w-2}{6-w} \frac{w-1}{4-w}$
as estimates of the area and the volume.

\bigskip 
\section{Archimedes' Lemma and Implications}
\label{sec:3-Archimedes}

The classic version of Archimedes' Lemma says that the area of
a slice of width $h$ of the $2$-dimensional sphere with radius $r$ is $2\pi r h$.
Equivalently, dropping a point onto a $2$-sphere uniformly at random and then projecting
it orthogonally to a diameter is equivalent to just dropping the point uniformly
at random onto the line segment.
Similarly, we recall the concept of a \emph{stationary Poisson point process
with intensity $\density$}:
the \emph{Poisson measure} of a Borel set is defined
to be $\density$ times the Lebesgue measure of the set,
and the points are sampled in such a way that the expected number in every Borel set
is its Poisson measure; see \cite{ScWe08} for the complete definition.
We state the interpretation of Archimedes' Lemma for uniform distributions
and Poisson point processes as a lemma:
\begin{lemma}[Archimedes]
  (1) The orthogonal projection of the uniform distribution
  on $\Sphere \subseteq \RR^3$ onto any given diameter
  is the uniform distribution on this line segment.
  (2) The orthogonal projection of the stationary Poisson point process
  with intensity $\density$ on $\Sphere$ onto any given diameter
  is the stationary Poisson point process with intensity $2 \pi \density$
  on this line segment.
  \qed
\end{lemma}
We need a few auxiliary statements for our proofs.
We can obtain them in different ways, including direct integration,
but we prefer the more illustrative application of Archimedes' Lemma.
\begin{lemma}[Expected Projection]
  \label{lem:expected_projection}
  For a line segment with endpoints $a, b \in \RR^3$,
  the expected length of the orthogonal projection onto a
  random direction is half the distance between $a$ and $b$.
\end{lemma}
\begin{proof}
  Assume without loss of generality that $a = 0$ is the origin of $\RR^3$,
  and $b$ has unit distance from $a$.
  The orthogonal projection of the connecting line segment onto
  a random direction with unit vector $e$ has length $\scalprod{b-a}{e}$,
  which is also the length of the orthogonal projection of $e$ onto the
  direction of $b-a$.
  Thus the average length of the projection is the distance to $0$
  of the projection of a random point on the unit sphere onto $b-a$.
  By Archimedes' Lemma, the projection is uniform, so the expected distance
  is $\tfrac{1}{2}$ or, in the general case, half the length of $b-a$.
\end{proof}

The next lemma uses the previous one to get the expected length of a random chord.
\begin{lemma}[Expected Distance]
  \label{lem:expected_distance}
  The expected Euclidean distance between two uniformly and independently
  chosen points on the unit sphere in $\RR^3$ is $4/3$.
\end{lemma}
\begin{proof}
  Call the points $a, b \in \Sphere$ and project them orthogonally onto
  a fixed diameter of the sphere.
  By Archimedes' Lemma, the projections are uniformly distributed on this diameter.
  The expected distance between two uniformly and independently chosen points
  on a line segment is one third of the length of the segment.
  This is easy to see, either by direct computation,
  or by gluing the ends of the segment and noticing that the experiment
  is equivalent to dropping three points onto a circle.
  Thus, the expected distance between the projection of $a$ and $b$ is $\tfrac{2}{3}$.
  Averaging over all diameters and applying Fubini's Theorem,
  we get that $\tfrac{2}{3}$ is half of the expected Euclidean distance
  between $a$ and $b$ by Lemma \ref{lem:expected_projection}.
\end{proof}

Consider three points dropped uniformly and independently onto
the unit circle $\SSS^1 \subseteq \RR^2$.
The probability that the triangle defined by the points is acute is $\tfrac{1}{4}$.
Perhaps the simplest argument was provided by Wendel \cite{Wen62}:
the central reflection of the points through the center of the circle
preserves the measure, and for each triple of points, two of the eight
possible reflections of a triangle (picking a point or its reflection) contain the center of the circle.
This argument does not generalize to triangles in higher dimensions:
central reflection through the center of the circumcircle no longer
preserves the measure.
Indeed, for triangles with vertices on $\Sphere$ the situation is already different.
\begin{lemma}[Acute Triangle]
  \label{lem:acute_triangle}
  The Euclidean triangle formed by three uniformly and independently
  chosen points on $\Sphere \subseteq \RR^3$ is acute with probability $\tfrac{1}{2}$.
\end{lemma}
\begin{proof}
  Let $a$, $b$, $c$ be the three vertices of the triangle.
  Since at most one angle of a triangle can be obtuse,
  it suffices to show that the angle at $a$
  is obtuse with probability $\tfrac{1}{6}$. 
  For any two points $a$ and $b$ on the sphere, the angle $\angle bac$ is obtuse
  if and only if the plane passing through the point $a$ and perpendicular to $b-a$
  separates $b$ and $c$; see the shaded area in Figure \ref{fig:angle}.
  The desired probability can thus be written as
  $\Expected{\One_{c \in \text{shaded region}}}$,
  which, after integrating $c$ out, is equal to the expected fraction
  of the area of the cap bounded by the plane.
  Archimedes' Lemma asserts, that this fraction is the ratio of the height of the cap
  to the diameter.
  The height equals $1 - \tfrac{\Edist{a}{b}}{2}$,
  so the ratio is $\tfrac{1}{2} - \tfrac{\Edist{a}{b}}{4}$.
  By Lemma \ref{lem:expected_distance}, the expected value of
  this ratio is $\tfrac{1}{6}$, which completes the proof.
\end{proof}	
\begin{figure}[ht]
  \includegraphics{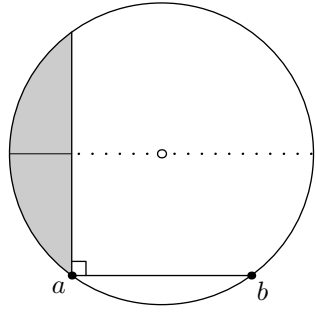}
  \caption{\footnotesize{The angle at vertex $a$ is obtuse iff $c$ lies in
    the shaded area beyond the plane with normal $b-a$ that passes through $a$.}}
  \label{fig:angle}
\end{figure}

\bigskip 
\section{Shape and Size}
\label{sec:4-ShapeandSize}

Now choose $n \geq 4$ points uniformly and independently on the sphere
and take their Euclidean convex hull.
With probability $1$, the points are in general position,
implying that the convex hull is a simplicial polytope.
The Euler formula and the integral geometric properties of the distribution
of such random polytopes facilitate the extension of Lemma \ref{lem:acute_triangle}.

\subsection{{\sc {Shape vs.\ Size}}}
\label{sec:41-ShapevsSize}

We begin with a corollary of the spherical Blaschke--Petkantschin formula
in \cite[Equation (2.1)]{EN18}, but see also \cite[Theorem 7]{Nik19},
namely that the shape of a random inscribed simplex is independent of its size.
\begin{lemma}[Shape vs.\ Size]
  \label{lem:shape_vs_size}
  Let $n \leq d$ points be uniformly and independently chosen from
  $\SSS^{d-1} \subseteq \RR^d$,
  which almost surely form an $(n-1)$-simplex and thus define a unique
  $(n - 2)$-dimensional circumsphere.
  Then the radius of this sphere is independent of the \emph{shape}
  of the simplex, i.e., the simplex scaled to unit circumradius.
\end{lemma}
\begin{proof}
  The spherical Blaschke--Petkantschin formula gives a decomposition
  of the measure on $n$-tuples of points on the sphere.
  Let $a_1, a_2, \ldots, a_n \in \SSS^{d-1}$,
  write $z$ and $r$ for the center and the radius of the
  $(n-2)$-dimensional circumsphere,
  and define $u_i = (a_i - z) / r$ for $1 \leq i \leq n$.
  Ignoring constant factors, the formula is
  \begin{align}
    \label{eqn:bp-small}
    \diff \PP_{d}(a_1, a_2, \ldots, a_n) &=
      c \cdot \mathrm{Vol}_{n-1}^{d-n+1} (\conv{\{u_1, u_2, \ldots, u_n\}})
      \diff \PP_{n-1}(u_1, u_2, \ldots, u_n) \otimes \diff g_{n, d}(r).
  \end{align}
  On the left, we have the measure $\PP_{d}$ on $n$-tuples
  of points on $\SSS^{d-1}$, and on the right the measure $\PP_{n-1}$
  on $n$-tuples of points on $\SSS^{n-2}$.
  Further, $g_{n, d}$ is a relatively complicated but explicit measure
  on the real line,
  and $\mathrm{Vol}_{n-1}$ denotes the $(n-1)$-dimensional volume of the simplex.
  Since $r$ and the $u_i$ appear in different factors,
  the distributions are independent.
\end{proof}

More precisely, the lemma states that the conditional probability of seeing a simplex,
conditioned on its circumsphere, is proportional to some power of its volume.
This implies that the distribution on the circle, induced by restricting the uniformly random triangle on $\Sphere$
to its circumcircle, is not uniform.
In particular, the conditional probability of an acute triangle
equals the probability of an acute triangle in $\Sphere$,
which is $1/2$ and thus double the probability for picking the vertices uniformly
along the circle.

\subsection{{\sc {Random Triangles}}}
\label{sec:42-RandomTriangles}

With this observation, we are ready to generalize Lemma \ref{lem:acute_triangle}.
\begin{theorem}[Random Triangle]
  \label{thm:random_triangle}
  Let $n \geq 3$ points be chosen uniformly and independently on the unit sphere
  in $\RR^3$, and let $X_n$ be their convex hull.
  Then a uniformly chosen random facet of $X_n$ is an acute triangle
  with probability $\tfrac{1}{2}$.
\end{theorem}	

\begin{proof}
  The case $n=3$ has already been proved, so we assume $n \geq 4$.
  By Euler's formula, all simplicial polytopes with $n$ vertices
  have the same number of facets, namely $f = 2n - 4$.
  We choose any three points on the sphere and condition on the event
  that the polytope has these three points as vertices.
  The probability, that the points form a facet of the random polytope
  depends only on the circumradius of the triangle spanned by them.
  Indeed, the requirement is equivalent to having all other points
  contained in only one of the two spherical caps determined by the triangle,
  and the probability of this event is a function of the areas of the caps,
  which in turn are functions of the circumradius.
  Further, the probability that a given facet of a polytope is the chosen one
  is~$\tfrac{1}{f}$, which is a constant.
  By the previous lemma, the circumradius is independent of whether or not
  the triangle is acute.
  Also, being acute or obtuse is clearly independent of the position of the other points.
  These independencies allow us to conclude that being a facet of
  a random inscribed polytope is independent of being acute,
  so Lemma \ref{lem:acute_triangle} implies that the probability
  of being acute is indeed~$\tfrac{1}{2}$.
\end{proof}

This theorem is aligned with the previous work on the topic.
Miles in \cite{Mil70} showed that for a Delaunay triangulation
--- which is the Euclidean space analogue of the convex hull ---
of a Poisson point process in the plane, half of the triangles are acute ergodically.
This has been transferred to the identical limiting statement for a sphere in \cite{EN18}.
The current theorem removes the asymptotic limit from the statement,
showing that the behavior is the same for a finite number of points.

\subsection{{\sc {Measure of Facets}}}
\label{sec:43-MeasureofFacets}

Another way of looking at the Blaschke--Petkantschin formula
gives an interpretation of the measure on the facets of
a random inscribed polytope.
Three points, $\theta$, define a facet
if all other sampled points lie on one side of the plane spanned by $\theta$.
With probability $1$, this plane splits the sphere into the two unequal \emph{caps}:
the \emph{small circumcap}, $\Scc{\theta}$,
on the side of the plane that does not contain the center of the sphere,
and the \emph{big circumcap}, $\Bcc{\theta}$, on the other side of the plane.
Call $\theta$ and the facet it defines \emph{small} or \emph{big},
depending which of the two caps is empty.
Only one of them can be empty, unless $n = 3$,
in which case the triangle is double-covered,
by one big and one small facet.
Let $\Theta = (\Sphere)^3$ be the set of ordered triangles in $\Sphere$,
and let $f_{sm}$ be the intensity measure on $\Theta$ of small facets.
According to the Blaschke--Petkantschin formula,
it is absolutely continuous with respect to the Lebesgue measure on $\Theta$:
\begin{align}
  \diff f_{sm}(\theta) 
    &=  \left(\frac{\Area(\Bcc{\theta})}{\Area(\Sphere)}\right)^{n-3}
        \frac{n(n-1)(n-2)}{\Area(\Sphere)^3} \diff \theta.
\end{align}
An analogous formula holds for $f_{bg}$, the intensity measure of big facets.
The area of the circumcap of $\theta$ depends only on the circumradius,
so the spherical Blaschke--Petkantschin formula gives a representation
of this measure as a product of measures like in \eqref{eqn:bp-small}:
$\diff f_{sm}(\theta) = f_1(r) \diff r \, f_2(s) \diff s$,
in which $f_1(r)$ is the distribution of the circumradius of a random triangle,
and $f_2(s)$ is the distribution of its shape.
This decomposition is useful in computing the expectation of any quantity
that depends on the shape and the radius in a multiplicative way,
such as the area, the volume, the total edge length, etc.
As an example,
writing $h(r) = \sqrt{1-r^2}$ for the height of the pyramid over the facet,
we get the volume of the polytope and its expectation:
\begin{align}
  \Vol(X_n)  &=  \sum\nolimits_{\theta\text{ small}}
                   \tfrac{1}{3} \Area(s) \, r^2  h(r)
               - \sum\nolimits_{\theta\text{ big}}
                   \tfrac{1}{3}\Area(s) \, r^2  h(r), \\
  \Expected{\Vol(X_n)}  &=  \tfrac{1}{3!}
    \int\nolimits_{\theta \in \Theta} \tfrac{1}{3} \Area(s)
        \, r^2 h(r) \; [\diff f_{sm}(\theta) - \diff f_{bg}(\theta)],
  \label{eqn:volume_expectation}
\end{align}
where the Blaschke--Petkantschin decomposition can be applied to compute the integral.
Note that the same observation applies for the Poisson case,
and it generalizes to any dimension.
We refer to the proof of Theorem \ref{thm:total_edge_length} as an application
of this viewpoint.

\bigskip 
\section{Intrinsic Volumes}
\label{sec:5-IntrinsicVolume}

This section is devoted to the expected intrinsic volumes
of a random polytope inscribed in $\Sphere$.
The recurrent integral expressions for these quantities have been computed in the uniform case
for a random convex hull inside a ball \cite{BuMue84},
they have been extended for the spherical case in \cite{buchta1985stochastical},
and the asymptotic was established in \cite{Muel90}.
Integral expressions for the Poisson case were developed in~\cite{Tem19}.

\subsection{{\sc {Uniform Distribution}}}
\label{sec:51-Uniform}

We give precise formulas for $n$ points sampled uniformly at random on $\Sphere$
and notice the special relation of the intrinsic volumes of their convex hull
to the intrinsic volumes of the ball.
We present proofs based on Crofton's formula and mention that the general approach
outlined in Section \ref{sec:43-MeasureofFacets} could also be used; see \cite{Kab17}.
We begin with the mean width.
\begin{theorem}[Mean Width I]
  \label{thm:MW_I}
  Let $n \geq 1$ points be chosen uniformly and independently on the unit sphere
  in $\RR^3$, and let $X_n$ be their convex hull. Then the expected mean width of $X_n$ is
  \begin{align}
    \Expected{\Width(X_n)}  &=  \Width(\BB^3) \cdot \tfrac{n-1}{n+1},
      \label{eqn:Thm-EIV-I1}
  \end{align}
  in which $\Width(\BB^3) = 2$ is the mean width of the unit ball.
\end{theorem}	
\begin{proof}
  The formula follows from Archimedes' Lemma.
  Indeed, by rotational symmetry, the mean width is the expected length
  of the projection of the random polytope onto a fixed direction.
  By Archimedes' Lemma, the projection is distributed as the
  segment connecting the first and last of $n$ points
  chosen uniformly and independently from $[-1, 1]$.
  Like in the proof of Lemma \ref{lem:expected_distance},
  we note that $n$ points divide a segment into $n+1$ identically
  (though not independently) distributed pieces, so the expected distance between
  the first and the last point is $2\tfrac{n-1}{n+1}$, as claimed.
\end{proof}

We now move to the area.
We start with a lemma that somehow escaped from Section \ref{sec:3-Archimedes}
to the third millennium.
We need the \emph{Crofton measure} on the space of lines in $\RR^d$,
which is the unique isometry-invariant measure on lines in $\RR^d$ normalized to have the total measure 1 of lines intersecting the unit ball.
In vague terms, it can be obtained by choosing a uniformly random direction,
followed by assigning the measure of the lines parallel to this direction
to be the Lebesgue measure on the orthogonal plane.
It is interesting, that in $\RR^3$ the Crofton measure has a simple description.
\begin{lemma}[Random Chords]
  \label{lem:random_chords}
  The probability distribution on lines intersecting $\Sphere$,
  defined by choosing two points uniformly and independently on $\Sphere$,
  coincides with the Crofton measure.
\end{lemma}
\begin{proof}
  A line that intersects $\Sphere$ in two points defines a \emph{chord},
  which is the straight segment connecting the two points.
  We compare the lengths of the chords under the two distributions.
  Since the two distributions are invariant under rotations,
  showing that both give rise
  to identical distributions of chord lengths suffices to prove the lemma.
	
  When we choose two random points, we can assume that one of them
  is the north pole, $N$.
  Then, for $\ell \in [0, 2]$, the probability that the second point, $x$,
  is closer to $N$ than $\ell$ equals the fraction of the sphere covered
  by the spherical cap centered at $N$, such that the furthest point
  of the cap has Euclidean distance $\ell$ to $N$.
  It is easy to see that the height of this cap is $\ell^2 / 2$.
  Thus, Archimedes' lemma implies that the fraction in question
  is ${\ell^2}/{4}$.
   	
  For the Crofton measure, the lines that intersect the ball in a chord
  of length less than $\ell$ are the ones that avoid the ball of radius
  $\sqrt{1-(\ell/2)^2}$ centered at the origin.
  For any fixed direction, the ratio of such lines to the measure of lines
  intersecting the ball is the area fraction of the annulus with
  inner radius $\sqrt{1-(\ell/2)^2}$ and outer radius $1$.
  This ratio is again $\ell^2/4$, which concludes the proof.
\end{proof}
The \emph{Crofton's formula} asserts
that the $(d-1)$-volume of the boundary of any convex body in~$\RR^d$
is proportional to the Crofton measure of the lines intersecting it.
Applying this in $\RR^3$, we see that the ratio of the area of the inscribed
polytope $X_n$ to the area of $\Sphere$ is the fraction of the lines
intersecting $\Sphere$ that also intersect $X_n$.
This observation lets us conclude the theorem.
\begin{theorem}[Area I]
  \label{thm:area_I}
  Let $n \geq 4$ points be chosen uniformly and independently on the unit sphere
  in $\RR^3$, and let $X_n$ be their convex hull.
  Then the expected surface area of $X_n$ is
  \begin{align}
    \Expected{\Area(X_n)}   &=  \;\Area(\BB^3) \cdot \tfrac{n-1}{n+1} \tfrac{n-2}{n+2},
      \label{eqn:Thm-EIV-I2}
  \end{align}
  in which $\Area(\BB^3) = 4 \pi$ is the area of the boundary of the unit ball.
\end{theorem}	
\begin{proof}
  By Lemma \ref{lem:random_chords}, the mentioned fraction
  is the probability that a random chord---which has the distribution
  of $X_2$---intersects $X_n$.
  Joining all points together, it is the probability that the \emph{extra}
  two points span a diagonal of $X_{n+2}$.
  There are $\frac{1}{2} (n+2)(n+1)$ pairs of vertices and
  (by Euler's formula) $3n$ edges, so this probability is
  \begin{align}
    \frac{\frac{1}{2} (n+2)(n+1) - 3n}{\frac{1}{2} (n+1)(n+2)}
      &=  \frac{(n-1)(n-2)}{(n+1)(n+2)}.
  \end{align}
  Multiplying by the area of $\Sphere$, we get the claimed identity.
\end{proof}

For the volume, we present a combinatorial proof without going into the
integral geometry details and refer the reader to \cite[Corollary 3.11]{Kab17}
for an alternative proof.
\begin{theorem}[Volume I]
  \label{thm:volume_I}
  Let $n \geq 4$ points be chosen uniformly and independently on the unit sphere
  in $\RR^3$, and let $X_n$ be their convex hull.
  Then the expected volume of $X_n$ is
  \begin{align}
    \Expected{\Vol(X_n)}    &=  \;\Vol(\BB^3) \cdot \tfrac{n-1}{n+1} \tfrac{n-2}{n+2}
                                             \tfrac{n-3}{n+3},
      \label{eqn:Thm-EIV-I3}
  \end{align}
  in which $\Vol(\BB^3) = \tfrac{4 \pi}{3}$ is the volume of the unit ball.
\end{theorem}
\begin{proof}
  The idea is similar to the proof of Theorem \ref{thm:area_I} but less direct.
  To write an integral geometry formula for the volume of the tetrahedron
  with base $abc$ and height $h$,
  we note that $2 \Area (abc) / 4 \pi$ is the fraction of lines
  intersecting $\Sphere$ that also intersect the triangle,
  and $\tfrac{h}{2}$ is the fraction of points on the diameter normal to $abc$
  for which the plane parallel to the triangle intersects the tetrahedron.
  Relating this formalism to the volume of the inscribed polytope,
  we pick a vertex $z$ as apex and form tetrahedra by connecting $z$
  to all triangular facets not incident to $z$.
  The total volume of these tetrahedra is $\Vol (X_n)$.
  Taking the sum over all vertices $z \in X_n$,
  we get the $2n-4$ triangles connected to $n-3$ vertices each,
  which amounts to $(2n-4)(n-3)$ tetrahedra with total volume $n \Vol (X_n)$.

  The rest of the argument is combinatorial.
  Picking $n+3$ points on $\Sphere$, we use $n$ to define the polytope,
  $2$ to define the line, and keep the remaining $1$ point to construct the plane.
  There are $\tfrac{1}{2} (n+3) (n+2) (n+1)$ ways to partition $X_{n+3}$
  into $X_n$, $X_2$, $X_1$.
  The plane that contains a facet of $X_n$ bounds two half-spaces,
  and we call the one that contains $X_n$ the \emph{positive side},
  while the other is the \emph{negative side} of the facet.
  Note that $X_1$ is on the negative side iff the facet of $X_n$ is not a facet
  of $X_{n+1} = \conv{(X_n \cup X_1)}$.
  We measure the volume of the $(2n-4)(n-3)$ tetrahedra combinatorially,
  and we do this for all partitions of the $n+3$ points simultaneously.
  Specifically, for each tetrahedron $zabc$, we multiply the number of lines
  that intersect $abc$ with the number of planes parallel to $abc$
  that intersect the tetrahedron.
  If $X_1$ is on the negative side of $abc$, then the product vanishes,
  so we can focus on the remaining facets,
  which are also the facets of $X_{n+1}$.

  Consider a facet, $F$, of $X_{n+1}$ that intersects $X_2$.
  There are $n-2$ choices for $X_1$, namely all vertices of $X_{n+1}$
  that are not incident to $F$.
  Picking $2$ of the $n-2$ vertices---one for $X_1$ and the other
  for the apex of the tetrahedron---we note that for one of
  two ordered choices the plane parallel to $F$ intersects the tetrahedron.
  This gives a total of $\binom{n-2}{2}$ plane-tetrahedron intersections,
  computed for a fixed choice of $X_{n+1}$ and $X_2$ and for a fixed triangle.
  Importantly, this number depends only on $n$.
  Next we recall Lemma \ref{lem:random_chords}, which asserts that
  $X_2$ gives a uniform measure on the lines intersecting~$\Sphere$.
  Each partition of the $n+3$ points into $X_{n+1}$ and $X_2$ gives a line
  that either intersects two facets (namely when $X_2$ is not an edge
  of $X_{n+3}$) or no facet (when $X_2$ is an edge of $X_{n+3}$).
  As argued in the proof of the area case above,
  of the $\binom{n+3}{2}$ pairs there are $3(n+3)-6 = 3n+3$ edges.
  The total number of line-triangle intersections is therefore
  $2 \binom{n+3}{2} - 2 (3n+3) = 2 \binom{n}{2}$,
  which again depends only on $n$.
  Multiplying with the number of plane-tetrahedron intersections,
  and averaging over all partitions of the $n+3$ points, we get
  \begin{align}
    \frac{2 \binom{n}{2} \binom{n-2}{2}}{\frac{1}{2} (n+3)(n+2)(n+1)}
      &=  \frac{n(n-1)(n-2)(n-3)}{(n+1)(n+2)(n+3)} ,
  \end{align}
  which is $n$ times the volume fraction, as required.
\end{proof}

\subsection*{Remark}
If we declare that the convex hull of three points is a double covered triangle,
then the formula \eqref{eqn:Thm-EIV-I2} holds for $n \geq 3$.
With this stipulation, the formulas for the intrinsic volumes
$V_1 = \Width$, $V_2 = \Area$ and $V_3 = \Vol$ can be
combined in a single expression that holds for all $n \geq 1$:
\begin{align}
  \Expected{V_k(X_n)}  &=  V_k(\BB^3) \cdot
    \frac{\Gamma(n)}{\Gamma(n - k)} \frac{\Gamma(n + 1)}{\Gamma(n + k + 1)}.
  \label{eq:all_uni}
\end{align}

\subsection{{\sc{Centrally Symmetric Polytopes}}}
\label{sec:52-CentrallySymmetricPolytopes}

We extend the analysis to centrally symmetric polytopes inscribed
in the unit sphere, reproving with combinatorial arguments the formulas first obtained in \cite{Kab17}.
To construct a random such polytope, we drop $n$ points uniformly and independently
on $\Sphere$ and take the convex hull of these points as well as their antibodes:
$\Xsym_{2n} = \conv{(X_n \cup (- X_n))}$.
\begin{theorem}[Intrinsic Volumes II]
  \label{thm:intrinsic_volumes_II}
  Let $n \geq 3$ points be chosen uniformly and independently on the unit sphere
  in $\RR^3$.
  Then the expected intrinsic volumes of $\Xsym_{2n}$ are
  \begin{align}
    \Expected{\Width(\Xsym_{2n})} &= \Width(\BB^3) \cdot \tfrac{n}{n+1} ,
      \label{eqn:cspW} \\
    \Expected{\Area(\Xsym_{2n})}  &= \;\Area(\BB^3) \cdot \tfrac{n-1}{n+2} ,
      \label{eqn:cspA} \\
    \Expected{\Vol(\Xsym_{2n})}  &= \;\Vol(\BB^3) \cdot \tfrac{n}{n+1} \tfrac{n-2}{n+3}.
      \label{eqn:cspV}
  \end{align}
\end{theorem}
\begin{proof}
  First the mean width.
  Dropping $n$ points into $[-1,1]$ and adding their reflections across $0$
  is equivalent to choosing $n$ points in $[0,1]$ and adding their negatives
  in $[-1,0]$.
  The expected distance between the first and the last point is
  $2 - \frac{2}{n+1} = 2 \frac{n}{n+1}$, which proves \eqref{eqn:cspW}.

  \medskip
  Second the area.
  Consider $n+2$ random pairs of antipodal points,
  which we divide into the vertices of $\Xsym_{2n}$,
  and an ordered quadruplet, $(a, b, -a, -b)$, forming the vertices of $\Xsym_4$.
  We use the latter to define a uniformly random line that intersects $\Sphere$.
  The probability that this line intersects $\Xsym_{2n}$ is the fraction of
  non-antipodal diagonals of $\Xsym_{2n+4}$ among the non-antipodal vertex pairs.
  The number of such pairs is $\frac{1}{2} (2n+4) (2n+2)$,
  from which we subtract the $3 (2n+4) - 6 = 6n+6$ edges of $\Xsym_{2n+4}$.
  The fraction is
  \begin{align}
    \frac{2 (n+2)(n+1) - 6 (n+1)}{2 (n+2)(n+1)}  &=  \frac{n-1}{n+2} .
  \end{align}
  Accordingly, the expected area of $\Xsym_{2n}$ is $\Area(\BB^3) = 4\pi$
  times this fraction, which proves \eqref{eqn:cspA}.

  \medskip
  Third the volume.
  We modify the proof of Theorem \ref{thm:volume_I} by working with only
  one set of tetrahedra, constructed by connecting the origin with the
  facets of the centrally symmetric polytope.
  To compute their total volume, we consider $n+3$ antipodal point pairs,
  which we divide into $\Xsym_{2n}$, $\Xsym_4$, $\Xsym_2$.
  As before, we use $\Xsym_4$ and $\Xsym_2$ to encode a line and a point,
  which we use to measure volume.
  The line defined by $\Xsym_4$ intersects either
  two or zero facets of $\Xsym_{2n+2} = \conv{(\Xsym_{2n} \cup \Xsym_{2})}$.
  For half of the intersected facets, the plane parallel to the facet
  that passes through the point defined by $\Xsym_2$ intersects the
  corresponding tetrahedron.
  The reason is that the plane intersects exactly one of the two tetrahedra
  spanned by the facet and its antipodal copy.
  The expected volume is therefore the fraction of non-antipodal diagonals
  of $\Xsym_{2n+6}$ among the non-antipodal vertex pairs, times the fraction
  of vertices that are incident neither to the facet nor its antipode:
  \begin{align}
    \frac{2 (n+3)(n+2) - 6 (n+2)}{2 (n+3)(n+2)} \cdot \frac{n-2}{n+1}
      &=  \frac{n}{n+3} \frac{n-2}{n+1} .
  \end{align}
  Accordingly, the expected volume of $\Xsym_{2n}$ is $\Vol(\BB^3) = \frac{4 \pi}{3}$
  times this fraction, which proves \eqref{eqn:cspV}.
\end{proof}

\subsection{{\sc {Poisson Point Process}}}
\label{sec:53-Poisson}

This subsection considers the same three intrinsic volumes but for a Poisson point
process rather than a uniform distribution on the $2$-sphere.
After proper rescaling, in the limit, the expected values for this process
should be the same as for the uniformly sampled points.
Here we give explicit expressions:
given a Poisson point process on $\Sphere$ of intensity $\rho$,
we write $X_\rho$ for its convex hull, and we study expected intrinsic volumes of this
random polytope. 
There are two ways of working with this case as well:
the general approach,
which uses Slivnyak--Mecke and Blaschke--Petkantschin formulas
(see the proof of Theorem \ref{thm:total_edge_length}),
and the reduction to the uniform distribution case,
which we employ in this section.
It uses the conditional representation of the Poisson point process
in a Borel set of finite measure $\lambda$:
first pick a random variable, $n_0$, from a Poisson distribution with
parameter $\lambda \rho$, and second sample
$n_0$ points independently and uniformly in the Borel set.
As such, all quantities of our interest can be written as
$\Expected{\cdot(X_\rho)}
    = \sum_{n = 0}^{\infty} \Expected{\cdot(X_n)} \Probable{n_0=n}$,
in which $\Probable{n_0=n} = e^{- 4 \pi \rho} \tfrac{(4 \pi \rho)^n}{n!}$
since the measure of the sphere is $\lambda = 4 \pi$.
To state the result, we recall the \emph{modified Bessel functions of the first kind}
defined for a real parameter, $\alpha$:
\begin{align}
  \BesselF{\alpha} (x)  &=  \tfrac{1}{\pi}
      \int_{\theta = 0}^{\pi} e^{x \cos \theta} \cos (\alpha \theta) \diff \theta
    - \frac{\sin (\alpha \pi)}{\pi}\int_{t=0}^\infty e^{-x \cosh t - \alpha t} \diff t ;
\end{align}
see e.g.\ \cite{OLBC10}.
The functions in this section all have explicit expressions, and can be expanded
using any mathematical software, but we keep them in
form of Bessel functions for uniformity.
To prepare the proof of Theorem \ref{thm:intrinsic_volumes_III},
we present a straightforward but technical computation of a specific series.
\begin{lemma}[Bessel representation]
  \label{lem:Bessel}
  For positive $k$,
  \begin{align}
    \sum\nolimits_{m = 0}^{\infty}
       \frac{\Gamma(m + k + 1)}{\Gamma(m + 2k + 2)}
       \frac{z^{m + k + 1}}{\Gamma(m + 1)}
    &=  e^{\frac{z}{2}} (\pi z)^{0.5} \BesselF{k+0.5} \left(\tfrac{z}{2}\right).
  \end{align}
\end{lemma}
\begin{proof}
  In addition to straightforward transformations on the expression,
  ({\small $*$}) uses the definition
  of the Kummer confluent hypergeometric function \cite{kummer1837deintefralibus},
  ({\small $**$}) uses the relation between modified Bessel and Kummer
  hypergeometric functions \cite[Formula 10.39.5]{OLBC10},
  and ({\small ${**}*$}) uses the Legendre Duplication Formula
  $\Gamma(2k + 2) = 2^{2k + 1} \Gamma(k) \Gamma(k + 1.5) \pi^{-0.5}$
  \cite[Formula 5.5.5]{OLBC10}:
  \begin{align}
    \sum\nolimits_{m = 0}^{\infty} &
      \frac{\Gamma(m + k + 1)}{\Gamma(m + 2k + 2)}
      \frac{z^{m + k + 1}}{\Gamma(m + 1)} 
    =  z^{k + 1} \frac{\Gamma(k + 1)}{\Gamma(2k + 2)}
         \sum\nolimits_{m = 0}^{\infty}
         \frac{\Gamma(m + k + 1)}{\Gamma(k + 1)}
         \frac{\Gamma(2k + 2)}{\Gamma(m + 2k + 2)} \frac{z^m}{m!} \\
    &\mathop{=}\limits^{(*)} z^{k + 1}
         \frac{\Gamma(k + 1)}{\Gamma(2k + 2)} \Hypergeometric{1}{1}{k+1}{2k+2}{z} \\
    &= \left[\frac{(z/2)^{k + \frac{1}{2}}}{2^{k + \frac{1}{2}} e^{\frac{z}{2}}
         \Gamma(k + 1.5)} \Hypergeometric{1}{1}{k+1}{2k+2}{z} \right]
         \frac{2^{2k + 1} e^{\frac{z}{2}} \Gamma(k + 1.5) z^{0.5} \Gamma(k + 1)}
              {\Gamma(2k + 2)}  \\
    &\mathop{=}\limits^{(**)} \BesselF{k+0.5} \left(\tfrac{z}{2}\right)
         \frac{2^{2k + 1} e^{\frac{z}{2}} \Gamma(k + 1.5) z^{0.5} \Gamma(k + 1)}
              {\Gamma(2k + 2)} \\
    &\mathop{=}\limits^{(***)} e^{\frac{z}{2}} (\pi z)^{0.5} \BesselF{k+0.5}
         \left(\tfrac{z}{2}\right). \qedhere
  \end{align}	
\end{proof}
Having this prepared, the following theorem is easy to prove.
\begin{theorem}[Intrinsic Volumes III]
  \label{thm:intrinsic_volumes_III}
  Let $X_\density$ be the convex hull of the stationary Poisson point process
  with intensity $\density > 0$ on the unit sphere in $\RR^3$.
  Writing $\Width(X_\density)$, $\Area(X_\density)$, and $\Vol(X_\density)$
  for the mean width, surface area, and volume, we obtain the following expressions
  for their expectations:
  \begin{align}
    \Expected{\Width(X_\density)} &= \Width(\BB^3)
      \cdot 2 \pi \density^{0.5} e^{-2 \pi \density} {\BesselF{1.5} (2\pi\density)},
      \label{eqn:Thm-EIV-II1} \\
    \Expected{\Area(X_\density)}  &= \; \Area(\BB^3)
      \cdot 2 \pi \density^{0.5} e^{-2 \pi \density} {\BesselF{2.5} (2\pi\density)},
      \label{eqn:Thm-EIV-II2} \\
    \Expected{\Vol(X_\density)}   &= \; \Vol(\BB^3)
      \cdot 2 \pi \density^{0.5} e^{-2 \pi \density} {\BesselF{3.5} (2\pi\density)},
      \label{eqn:Thm-EIV-II3}
  \end{align}
  in which $\BesselF{\alpha} (x)$ is the modified Bessel function of the first kind.
\end{theorem}
\subsection*{Remark}
As expected, the factors after the intrinsic volumes
of $\BB^3$ tend to $1$ when $\rho \to \infty$.
\begin{proof}
  According to the conditional representation of a Poisson point process,
  it suffices to compute the sum of a series with terms
  from the uniform case \eqref{eq:all_uni}.
  We can thus write
  \begin{align}
    \Expected{V_k(X_\rho)} &= V_k(\BB^3) e^{- 4 \pi \rho}
      \sum\nolimits_{n = k + 1}^{\infty} 
      \frac{\Gamma(n)}{\Gamma(n - k)} \frac{\Gamma(n + 1)}{\Gamma(n + k + 1)}
      \frac{(4 \pi \rho))^n}{n!}.
  \end{align}
  Now we do a simple substitution, $m = n - k - 1$,
  and use the identity $\Gamma(n + 1) = n!$ to get into the setting of
  Lemma \ref{lem:Bessel} with $z = 4\pi \rho$:
  \begin{align}
    \Expected{V_k(X_\rho)} &= V_k(\BB^3)  e^{- 4 \pi \rho}
      \sum\nolimits_{m = 0}^{\infty}
        \frac{\Gamma(m + k + 1)}{\Gamma(m + 1)\Gamma(m + 2k + 2)}
        {(4 \pi \rho)^{m + k + 1}} \\
    &= V_k(\BB^3) e^{-4\pi \rho} e^{2\pi\rho}
         (\pi \cdot 4\pi\rho)^{0.5} \BesselF{k + 0.5}(2\pi\rho) \\
    &= V_k(\BB^3) 2\pi \rho^{0.5} e^{-2\pi\rho} \BesselF{k + 0.5}(2\pi\rho). \qedhere
  \end{align}
\end{proof}

\bigskip 
\section{Length and Distance}
\label{sec:6-LengthandDistance}

In this section, we study two questions about expected length,
namely the total edge length of a random inscribed polytope and
the Euclidean distance to a fixed point.
The total edge length is not an intrinsic volume,
but the most generic version of the Blaschke--Petkantschin formula
can deal with almost any function of the polytope,
including the sum of edge lengths.
As in Section \ref{sec:5-IntrinsicVolume}, we consider both the uniform
distribution and the Poisson point process,
noting that the result in the latter case bears striking resemblance
to the formulas given in Theorem \ref{thm:intrinsic_volumes_III}.

\subsection{{\sc {Total Edge Length}}}
\label{sec:61-TotalEdgeLength}

We again prepare with a technical lemma.
\begin{lemma}
  \label{lem:integrals_I}
  We have
  \begin{align}
    J_\Length(n)  &=  \int_{t=0}^1 t^{3/2} (1-t)^{-1/2}
                      \left[ \left( \tfrac{1+\sqrt{1-t}}{2} \right)^{n-3}
                           + \left( \tfrac{1-\sqrt{1-t}}{2} \right)^{n-3}
                      \right] \diff t
                   =  32 \cdot \BetaF{n-\tfrac{1}{2}}{\tfrac{5}{2}};
									\label{eqn:Int-I3}\\
			K_\Length(\density)  &=  \int_{t=0}^1 t^{3/2} (1-t)^{-1/2}
									\left[ e^{-2 \pi \density (1+\sqrt{1-t})}
											 + e^{-2 \pi \density (1-\sqrt{1-t})} \right] \diff t
							 =  \tfrac{3}{2 \pi} \, \density^{-2} \, e^{-2 \pi \density}
																	\, \BesselF{2}(2 \pi \density) .
			\label{eqn:Int-II3}
  \end{align}
\end{lemma}
To get the right-hand side of \eqref{eqn:Int-I3}, we first apply a change
of variables $s = 1 + \sqrt{1-t}$ to the left term
and $s = 1 - \sqrt{1-t}$ to the right term or, equivalently,
$t = 2s - s^2$ to both.
Then writing $q = s/2$, we recognize the integral as a multiple of the
beta function for parameters $n - 0.5$ and $2.5$.
For \eqref{eqn:Int-II3}, we first use the same change of variables,
and then set $s = 1 + \cos\theta$ to arrive at
the expression of 10.32.2 in \cite{OLBC10}.
We leave the details to the reader, and note that the integrals can also be computed
with mathematical software.
\begin{theorem}[Total Edge Length]
  \label{thm:total_edge_length}
  Let $X_n$ be the convex hull of $n \geq 3$ points chosen uniformly and independently
  at random on $\Sphere$,
  and let $X_\density$ be the convex hull of a stationary Poisson point process
  with intensity $\density > 0$ on $\Sphere$.
  Then the sums of lengths of the edges on the two inscribed polytopes satisfy
  \begin{align}
    \Expected{\Length(X_n)}
      &= \tbinom{n}{3} \tfrac{512}{3\pi}
                       \cdot B\left(n-\tfrac{1}{2}, \tfrac{5}{2}\right)
     \;\;\;\;\;\;\;\;\;\;\;\;\, \left[ = \tfrac{64}{3 \sqrt{\pi}} \sqrt{n} \cdot (1+o(1))\right],
    \label{eqn:Thm-TEL-I} \\
    \Expected{\Length(X_\density)} 
      &=  \tfrac{128}{3} \density^{0.5} \cdot
          2 \pi \density^{0.5} e^{-2 \pi \density} {\BesselF{2} (2 \pi \density)}
     \;\;\; \left[ = \tfrac{64}{3 \sqrt{\pi}} \sqrt{4\pi\density} \cdot (1+o(1)) \right].
    \label{eqn:Thm-TEL-II}
  \end{align}
\end{theorem}	
\begin{proof}
  The arguments for the two random models are sufficiently similar,
  so we can present them in parallel,
  writing $X$ whenever a relation holds for both, $X_n$ and $X_\density$.
  We follow the strategy sketched in Section \ref{sec:43-MeasureofFacets}.
  Write $\Length (F)$ for the perimeter of a triangle $F$.
  Every edge belongs to two triangles, which implies that the total edge length
  satisfies $\Length(X) = \sum_{\{a,b,c\} \subseteq X}
                          \OneFacet{abc} \, \tfrac{1}{2} \Length(abc)$,
  where $\OneFacet{abc}$ is the indicator that $abc$ is a facet of $X$.
  Recall that the plane passing through $a, b, c$ cuts the sphere
  into two spherical caps, one big and the other small.
  Three points form a facet iff one of their circumcaps is empty.
  If the total number of points is at least $4$,
  the two caps cannot be empty simultaneously,
  so $\OneFacet{abc} = \OneEmpty{\Bcc{abc}} + \OneEmpty{\Scc{abc}}$,
  in which the indicators on the right-hand side of the equation
  sense if the caps are empty.
  If $X$ has only 3 points, we consider it to be a double cover with two facets,
  so the formula still make sense.
  Rewriting the total edge length in terms of the circumcaps
  and taking the expectation, we get
  \begin{align}
    \Expected{\Length(X_n)} 
      &=  \tbinom{n}{3} \tfrac{1}{2}
          \, \Expected{\OneEmpty{\Scc{abc}} + \OneEmpty{\Bcc{abc}}}
          \, \Length(abc) .
  \end{align}
  Rewriting the expectation, we get 
  \begin{align}
    \Expected{\Length(X)}
      &=  C \int_{a, b, c \in \Sphere}
            ( \Probable{\Scc{abc}\text{ empty}} + \Probable{\Bcc{abc} \text{ empty}} )
            \, \tfrac{1}{2} \Length(abc) \diff a \diff b \diff c ,
  \end{align}
  in which $X = X_n$ and $C = \tbinom{n}{3} /(4 \pi)^3$.
  Using the Slivnyak--Mecke formula, we get the same relation for $X = X_\density$
  except that $C = {\density^3} / {3!}$.
  Call the Euclidean radius of the circle passing through $a, b, c$ the
  (common) \emph{radius} of $\Scc{abc}$ and $\Bcc{abc}$,
  and write $P_+(r)$ for the probability that one of the two caps
  of radius $r$ is empty.
  We apply the Blaschke--Petkantschin formula to get
  \begin{align}
    \Expected{\Length(X)}  &=  C \cdot 2 \pi
         \int_{t = 0}^{1} t (1-t)^{-1/2} \int_{u, v, w \in \SSS^1}
           P_+ (\sqrt{t}) \, \tfrac{1}{2} \Length(\sqrt{t} \cdot uvw) 2! \Area(uvw)
           \diff u \diff v \diff w \diff t 
     \label{eqn:int_l_1} \\
                           &=  C \cdot 2 \pi
           \int_{t = 0}^{1} t^{3/2} (1-t)^{-1/2} P_+ (\sqrt{t}) \diff t
           \int_{u, v, w \in \SSS^1} \Area(uvw) \Length(uvw)
           \diff u \diff u \diff u ,
    \label{eqn:int_l_2}
  \end{align}
  with $C = \tbinom{n}{3} / (4 \pi)^3$ in the uniform distribution case,
  and $C = \density^3 / 3!$ in the Poisson point process case.
  Explicitly,
  \begin{align}
    P_+(n,r)  &=  (\tfrac{1+h}{2})^{n-3} + (\tfrac{1-h}{2})^{n-3},  \\
    P_+(\density, r)  &=  e^{-2 \pi \density (1 + h)}
                         +  e^{-2 \pi \density (1 - h)},
  \end{align}
  in which $h = \sqrt{1-r^2}$ so that $1-h$ and $1+h$ are the heights
  of the two caps.
  Plugging them into \eqref{eqn:int_l_2},
  we get the first integral on the right-hand side
  equal to $J_\Length (n)$ and to $K_\Length (\density)$, respectively;
  see Lemma \ref{lem:integrals_I}.
  To compute the second integral, we fix $u=(1,0)$ and parametrize $v, w$
  with their angles relative to $u$, which we denote $\alpha, \beta$.
  The integral of the area times the length is thus $2 \pi$ times
  the double integral over the two angles:
  \begin{align}
      2 \pi \int_{\alpha, \beta = 0}^{2\pi}
              \!\! \Area(\alpha, \beta)
                \Length(\alpha, \beta) \diff \beta \diff \alpha             
      &=  32 \pi \int_{\alpha, \beta = 0}^{2\pi}
            \!\! \left( \sin \tfrac{\alpha}{2} + \sin \tfrac{\beta}{2}
                 + |\sin \tfrac{\gamma}{2}| \right)
            \sin \tfrac{\alpha}{2} \sin \tfrac{\beta}{2} |\sin \tfrac{\gamma}{2}|
            \diff \beta \diff \alpha ,
    \label{eqn:length_last}
  \end{align}
  in which we use $\Length(\alpha, \beta) = U+V+W$
  and $\Area(\alpha, \beta) = \tfrac{1}{4} UVW$,
  with edges of length $U = 2 \sin \tfrac{\alpha}{2}$, $V = 2 \sin \tfrac{\beta}{2}$,
  and $W = 2 |\sin \tfrac{\gamma}{2}|$, where $\gamma = \alpha - \beta$,
  to get the right-hand side.
  Using the Mathematica software, we find that \eqref{eqn:length_last}
  evaluates to $\tfrac{512 \pi}{3}$.
  Combining the values, we get
  \begin{align}
    \Expected{\Length(X_n)}  &=  \tbinom{n}{3} \tfrac{512}{3\pi}
                                 \cdot B(n-\tfrac{1}{2}, \tfrac{5}{2}) ,   
    \label{eqn:length_I} \\
    \Expected{\Length(X_\density)}  &=  \Expected{\Length(X_n)}
                                        \cdot \tfrac{\density^3 (4 \pi)^3}{n(n-1)(n-2)}
                                        \cdot \tfrac{K_\Length(\density)}{J_\Length(n)} .
    \label{eqn:length_II}
  \end{align}
  The asymptotic expansion claimed in \eqref{eqn:Thm-TEL-I} can now
  be obtained from \eqref{eqn:length_I} using Mathematica.
  The relation claimed in \eqref{eqn:Thm-TEL-II} follows straightforwardly
  from \eqref{eqn:length_II}.
\end{proof}

\subsection*{Remark}
Like in Theorem \ref{thm:intrinsic_volumes_III}, it is also easy to obtain
\eqref{eqn:Thm-TEL-II} from \eqref{eqn:Thm-TEL-I} using
the conditional representation of the Poisson point process
and Lemma \ref{lem:Bessel}.

\subsection{{\sc Minimum Distance}}
\label{sec:62-MinimumDistance}

We finally study how close a random collection of points approaches
a fixed point on the unit $2$-sphere.
Somewhat surprisingly, there is a connection to the volumes of high-dimensional
unit balls.
To state the result, we write $\Vol(\BB^m)$ for the $m$-dimensional volume
of the unit ball in $\RR^m$.
\begin{theorem}[Minimum Distance]
  \label{thm:minimum_distance}
  Let $n$ points be chosen uniformly and independently on the unit sphere in $\RR^3$.
  Then the expected minimum Euclidean distance from a fixed point on the sphere
  is ${\Vol(\BB^{2n+1})} / {\Vol(\BB^{2n})}$.
\end{theorem}
\begin{proof}
  Let $N \in \Sphere$ be the fixed point and consider the cap of points
  with Euclidean distance at most $r$ from $N$.
  Equivalently, the spherical radius of the cap is $2\arcsin r/2$.
  Using Archimedes' Lemma, we get $r^2 \pi$ for the area of this cap.
  The probability that none of the $n$ points lie in this cap is therefore
  \begin{align}
    \label{eqn:pr-far}
    \Probable{R \geq r} 
      &=  \left( \tfrac{4\pi - \pi r^2}{4\pi} \right)^n
       =  \left(1 - \tfrac{r^2}{4} \right)^n ,
  \end{align}
  in which $R$ is the maximum Euclidean radius for the which the
  cap has no point in its interior.
  This maximum radius is the minimum distance to $N$,
  whose expectation we compute using the formula
  \begin{align}
    \Expected{R} 
       &=  \int_{r=0}^{\infty} \Probable{R \geq r} \diff r 
        =  \int_0^2 \left(1 - \tfrac{r^2}{4}\right)^n \diff r
        =  \int_{-1}^1 \frac{\Vol(\BB^{2n})}{\Vol(\BB^{2n})} (1 - t^2)^n \diff t 
        =  \frac{\Vol(\BB^{2n+1})}{\Vol(\BB^{2n})} ,
  \end{align}
  in which we get the ratio on the right by observing that the $2n$-dimensional volume
  of the slice of $\BB^{2n+1}$ at distance $t$ from the center is
  $\Vol(\BB^{2n})(1-t^2)^n$.
\end{proof}

We recall that the double factorial of an even positive integer is $(2n)!! = 2^n n!$
and that of an odd positive integer is $(2n+1)!! = (2n+1)!/(2n)!!$.
The volumes of the balls are $\Vol(\BB^{2n+1}) = 2^{n+1} \pi^n / (2n+1)!!$
and $\Vol(\BB^{2n}) = \pi^n / n!$.
It follows that the ratio is
\begin{align}
  \Expected{R} &=  \frac{\Vol(\BB^{2n+1})}{\Vol(\BB^{2n})} 
                =  \tfrac{2(2n)!!}{(2n+1)!!}
               \mathop{\thicksim}\limits_{n\to\infty}  \sqrt{\tfrac{\pi}{n}} ,
\end{align}
in which the final formula is obtained using Sterling's Formula
for factorials.
We can repeat the argument from Theorem \ref{thm:minimum_distance}
to get the expected minimum \emph{spherical} distance from $N$,
which we denote $\Phi$.
The probability that this distance exceeds a threshold is
$\Probable{\Phi \geq \phi}
  = (1 - \sin^2 \frac{\phi}{2})^n = \cos^{2n} \frac{\phi}{2}$.
The expected value of the minimum spherical distance is therefore
\begin{align}
  \Expected{\Phi} 
     &=  \int_{\phi=0}^{\pi} \cos^{2n} \tfrac{\phi}{2} \diff \phi 
        =  2 \int_{\phi=0}^{\pi/2} \cos^{2n} \phi \diff \phi
        =  \BetaF{n+\tfrac{1}{2}}{\tfrac{1}{2}}
        =  \frac{\pi (2n)!}{4^n (n!)^2}
        \mathop{\thicksim}\limits_{n\to\infty}  \sqrt{\tfrac{\pi}{n}} .
\end{align}
Similarly, we can get the higher moments of the minimum distance.
Returning to the Euclidean distance, and writing $s = r^2/4$,
we get the density of the distribution of $s$ from \eqref{eqn:pr-far}:
it is the negative of the derivative of $(1-s)^n$,
which is $n (1-s)^{n-1}$.
From this we get the $k$-th power of the minimum distance as
$r^k = 2^k s^{k/2}$:
\begin{align}
  \Expected{R^k}
    &=  \int_{s=0}^1 2^k s^{k/2} \, n (1-s)^{n-1} \diff s
     =  n 2^k \, \BetaF{n}{\tfrac{k}{2}+1}
    \mathop{\thicksim}\limits_{n\to\infty}  \sqrt{\tfrac{2 \pi k^k}{2^k n^k}} .
\end{align}

\bigskip 
\section{Deficiencies}
\label{sec:7-Deficiencies}

Since the random inscribed polytopes approximate the unit $3$-ball,
we compare their measures with that of the ball. 
Letting $\mu$ be a measure that applies to $\BB^3$ and to inscribed polytopes
alike, we call
\begin{align}
  \NDef{\mu(X_n)}  &=  1 - {\mu (X_n)}/{\mu (\BB^3)} 
\end{align}
the corresponding \emph{normalized deficiency}.
Besides the deficiency of a random inscribed polytope,
we consider the deficiency in the ideal regular case,
for what we call the \emph{virtual model}, $\model_n$.
Despite the construction in \cite{akopyan2015hexagonal},
there are no regular simplicial polytopes inscribed in $\Sphere$
other than for $n = 4, 6, 20$ vertices.
We therefore consider the regular spherical triangle of area
$a_n = \tfrac{4 \pi}{2n-4}$, tacitly ignoring the fact that for most $n$,
we cannot decompose the sphere into congruent copies of this triangle.
All three of its angles are equal,
namely $\alpha_n = (a_n + \pi)/3$, by Girard's Theorem.
We are interested in the corresponding Euclidean triangle.
\begin{lemma}[Euclidean Triangle]
  \label{lem:Euclidean_triangle}
  Consider two Euclidean triangles that share their four vertices with 
  two adjacent regular spherical triangles of area $4 \pi / (2n-4)$ each.
  The length of an edge, the area of a triangle, the volume of the tetrahedron
  connecting the Euclidean triangle to the origin, and the angle between
  the two normals are
  \begin{align}
    L_n  &=  \tfrac{2 \sqrt{2 \pi}}{\sqrt[4]{3}} \cdot \sqrt{\tfrac{1}{n}}
           + \tfrac{\sqrt{2 \pi} (18 - 5 \sqrt{3} \pi)}{9 \sqrt[4]{3}}
           \cdot \sqrt{ \tfrac{1}{n^3} } + O \left( \sqrt{\tfrac{1}{n^5}} \right) ,   
      \label{eqn:length-n} \\
    A_n  &=  2 \pi \cdot \tfrac{1}{n}
           + \tfrac{36 \pi - 10 \sqrt{3} \pi^2}{9} \cdot \tfrac{1}{n^2}
           + O( \tfrac{1}{n^3} ) ,
      \label{eqn:area-n} \\
    V_n  &=  \tfrac{2 \pi}{3} \cdot \tfrac{1}{n}
           + \tfrac{4 \pi - 2 \sqrt{3} \pi^2}{3}
           \cdot \tfrac{1}{n^2} + O( \tfrac{1}{n^3} ) ,
      \label{eqn:volume-n} \\
    \vartheta_n  &=  \tfrac{\sqrt{8 \pi} \sqrt[4]{3}}{3} \cdot \sqrt{\tfrac{1}{n}}
                + \tfrac{18 \sqrt{6 \pi} + 5 \pi \sqrt{2 \pi}}
                        {27 \sqrt[4]{3}} \cdot \sqrt{\tfrac{1}{n^3}}
                + O \left( \sqrt{\tfrac{1}{n^5}} \right).
      \label{eqn:angle-n}
  \end{align}
\end{lemma}
\noindent We omit the proof, which is straightforward but tedious.
As mentioned before, a convex polytope all of whose facets are regular triangles
does not exist for most $n$.
We nevertheless define the \emph{total edge length},
the \emph{area}, and the \emph{volume} of the virtual model as
$\Length(\model_n) = (3n-6) L_n$,
$\Area(\model_n) = (2n-4) A_n$,
and $\Vol(\model_n) = (2n-4) V_n$.
To get a similar definition of the \emph{mean width}, we recall it is
$\frac{1}{2 \pi}$ times the mean curvature,
and a convenient formula for the latter is the sum, over all edges,
of half the length times the angle between the outer normals of the two
incident faces:
$\Width(\model_n) = (3n-6) \frac{1}{4 \pi} L_n \vartheta_n$.
We conjecture that the mean width, area, and volume of the virtual model
are beyond the reach of convex inscribed polytopes:
\begin{conjecture}[Upper Bounds]
  \label{conj:upper_bounds}
  Let $X_n$ be the convex hull of $n \geq 4$ points on the unit sphere in~$\RR^3$.
  Then $\Width (X_n) \leq \Width (\model_n)$,
       $\Area (X_n) \leq \Area (\model_n)$,
  and  $\Vol (X_n) \leq \Vol (\model_n)$.
\end{conjecture}
Compare the inequalities in Conjecture \ref{conj:upper_bounds}
with \cite[Section 9]{FeTo64}.
The total edge length permits no such inequality.

\medskip
It is of some interest to probe how close or far from the virtual model
the random inscribed polytopes are.
To this end, we take a look at the ratio of deficiencies.
We will see shortly that the ratios of the mean width, the area, and the volume
converge to $1.984\ldots$, $1.984\ldots$, and $2.205\ldots$, respectively.
For the total edge length, we do not have deficiencies but we can compare
the lengths directly.
We get the expected normalized mean width deficiency of a random inscribed polytope
from \eqref{eqn:Thm-EIV-I1}, compute the normalized mean width of the virtual model
using \eqref{eqn:length-n}, \eqref{eqn:angle-n}, and look at the ratio to compare:
\begin{align}
  \Expected{\NDef{\Width(X_n)}}
    &=  1 - \frac{\Expected{\Width(X_n)}}{\Width(\BB^3)}
     =  1 - \tfrac{n-1}{n+1}
     =  2 \cdot \tfrac{1}{n} + O\left(\tfrac{1}{n^2}\right),        \\
  \NDef{\Width(\model_n)}
    &=  1 - \frac{\Width(\model_n)}{\Width(\BB^3)}
     =  \tfrac{5 \sqrt{3} \pi}{27} \cdot \tfrac{1}{n}
      + O\left(\tfrac{1}{n^2}\right),       \\
  \frac{\Expected{\NDef{\Width(X_n)}}}{\NDef{\Width(\model_n)}}
    &=  \tfrac{18 \sqrt{3}}{5 \pi} + O \left( \tfrac{1}{n} \right)
        \underset{n \to \infty}{\longrightarrow}  1.984\ldots.
\end{align}
We repeat the comparison for the area, using \eqref{eqn:Thm-EIV-I2}
and \eqref{eqn:area-n} to compute the normalized deficiencies:
\begin{align}
  \Expected{\NDef{\Area(X_n)}}
    &=  1 - \frac{\Expected{\Area(X_n)}}{\Area(\BB^3)}
     =  1 - \tfrac{(n-1)(n-2)}{(n+1)(n+2)}
     =  6 \cdot \tfrac{1}{n} + O\left(\tfrac{1}{n^2}\right),        \\
  \NDef{\Area(\model_n)}
    &=  1 - \frac{\Area(\model_n)}{\Area(\BB^3)}
     =  \tfrac{5 \sqrt{3} \pi}{9} \cdot \tfrac{1}{n}
        + O\left(\tfrac{1}{n^2}\right) ,                            \\
  \frac{\Expected{\NDef{\Area(X_n)}}}{\NDef{\Area(\model_n)}}
    &=  \tfrac{18 \sqrt{3}}{5 \pi} + O\left(\tfrac{1}{n}\right)
        \underset{n \to \infty}{\longrightarrow}  1.984\ldots.
\end{align}
We repeat the comparison for the volume, using \eqref{eqn:Thm-EIV-I3}
and \eqref{eqn:volume-n} to compute the normalized deficiencies:
\begin{align}
  \Expected{\NDef{\Vol(X_n)}}
    &=  1 - \frac{\Expected{\Vol(X_n)}}{\Vol(\BB^3)}
     =  1 - \tfrac{(n-1)(n-2)(n-3)}{(n+1)(n+2)(n+3)}
     =  12 \cdot \tfrac{1}{n} + O\left(\tfrac{1}{n^2}\right) ,         \\
  \NDef{\Vol(\model_n)}       
    &=  1 - \frac{\Vol(\model_n)}{\Vol(\BB^3)}
     =  \sqrt{3} \pi \cdot \tfrac{1}{n}
        + O\left(\tfrac{1}{n^2}\right) ,                                 \\
  \frac{\Expected{\NDef{\Vol(X_n)}}}{\NDef{\Vol(\model_n)}}
    &=  \tfrac{4 \sqrt{3}}{\pi} + O\left(\tfrac{1}{n}\right)
        \underset{n \to \infty}{\longrightarrow}  2.205\ldots.
\end{align}
We finally consider the total edge length.
Since $\Length(\BB^3)$ is not defined, we are not able to compute any deficiency.
Nevertheless, we can compare the total edge length of a random inscribed polytope,
which we get from \eqref{eqn:Thm-TEL-I},
with that of the model,
which we compute with \eqref{eqn:length-n}:
\begin{align}
  \frac{\Expected{\Length(X_n)}}{\sqrt{n}}
    &=  \tfrac{64}{3 \sqrt{\pi}} + O\left(\tfrac{1}{n}\right)
        \underset{n \to \infty}{\longrightarrow} 12.036\ldots ,     \\
  \frac{\Length(\model_n)}{\sqrt{n}}
    &=  \tfrac{6 \sqrt{2 \pi}}{\sqrt[4]{3}} + O\left(\tfrac{1}{n}\right)
        \underset{n \to \infty}{\longrightarrow} 11.427\ldots .
\end{align}
The ratio converges to $1.053\ldots$.
The fact that the model has smaller total edge length than the random
inscribed polytope suggests a nearby local minimum.
It can of course not be a global minimum because there are inscribed
polytopes with arbitrarily small total edge length for any number
of vertices.

\bigskip 
\section{Ellipsoid with Homeoid Density}
\label{sec:8-Ellipsoid}

In this section, we extend the expressions for the intrinsic volumes and total edge length
from the sphere to the ellipsoid.
On the latter, we consider the \emph{homeoid density}, which is the
push-forward of the uniform measure on $\Sphere$ under the linear transform,
$\Linear$, that sends the sphere to the ellipsoid.
It can also be defined as the limit of the uniform measure in the layer
between the ellipsoid and its concentrically scaled copy;
see \cite[Section 9.2]{Arn04}.
\begin{figure}[ht]
  \vspace{0.2in}
  \includegraphics{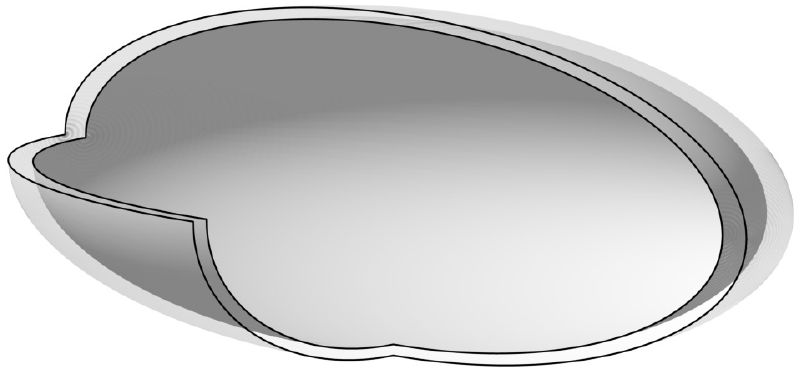}
  \caption{\footnotesize{Five eighths of the solid ellipsoid,
    and the layer between its boundary and the boundary of a scaled copy.}}
  \label{fig:ellipsoid}
\end{figure}
It follows from work of Newton and Ivory that in a charged metal shell,
electrons distribute according to this homeoid density.
This is the only distribution in which the electric field inside the shell vanishes and,
in addition, the level sets of the potential energy outside the shell are
confocal ellipsoids.

We write $\EE^3$ for the solid ellipsoid and $\partial \EE^3$ for its boundary;
that is: $\EE^3 = \Linear (\BB^3)$ and $\partial \EE^3 = \Linear (\Sphere)$.
Letting $p \geq q \geq r$ be the half-lengths of its axes,
we note that the volume of $\EE^3$ is $\tfrac{4 \pi}{3} \, p q r$.
There is no such simple expression for the area, but there are incomplete
elliptic functions of the first and second kind, $E$ and $F$, such that
\begin{align}
  \Area(\EE^3)  &=  2 \pi \left[ r^2 + \tfrac{q r^2}{\sqrt{p^2-r^2}}
    \, F \left( \! \sqrt{1-\tfrac{r^2}{p^2}};
                   \tfrac{p}{q} \sqrt{\tfrac{q^2-r^2}{p^2-r^2}} \right)
    + q \sqrt{p^2-r^2}
    \, E \left( \! \sqrt{1-\tfrac{r^2}{p^2}};
                  \tfrac{p}{q} \sqrt{\tfrac{q^2-r^2}{p^2-r^2} } \right) \right].
  \label{eq:surface area}
\end{align}
To get a formula for the mean width, we use a well known relation between $\EE^3$
and its dual ellipsoid, denoted $\DD^3$,
whose half-lengths are $\tfrac{1}{p}, \tfrac{1}{q}, \tfrac{1}{r}$, namely
$\Width(\EE^3) = \tfrac{pqr}{2\pi} \cdot \Area(\DD^3)$.
We refer to \cite[Prop. 4.8]{kabluchko2016intrinsic} for a formulation
of this relation and to \cite{petrov2019uniqueness} for an application in $\RR^3$.
We now generalize the theorems from Sections \ref{sec:5-IntrinsicVolume}
and \ref{sec:6-LengthandDistance} to state how the convex hull of a random
inscribed polytope approximates the intrinsic volumes of the ellipsoid.
\begin{theorem}[Inscribed in Ellipsoid]
  \label{thm:inscribed_in_ellipsoid}
  Let $n \geq 4$ points be chosen independently according to the homeoid distribution
  on $\partial \EE^3 \subseteq \RR^3$,
  and let $Y_n$ be their convex hull.
  The intrinsic volumes satisfy
  \begin{align}
    \Expected{\Width(Y_n)}  &=  \Width(\EE^3) \cdot \tfrac{n-1}{n+1},
      \label{eqn:Thm-EIV-III1} \\
    \Expected{\Area(Y_n)}   &=  \;\Area(\EE^3)
                                \cdot \tfrac{n-1}{n+1} \tfrac{n-2}{n+2},
      \label{eqn:Thm-EIV-III2} \\
    \Expected{\Vol(Y_n)}    &=  \;\Vol(\EE^3)
                                \cdot \tfrac{n-1}{n+1} \tfrac{n-2}{n+2} \tfrac{n-3}{n+3},
      \label{eqn:Thm-EIV-III3} 
  \end{align}
  and the expected total edge length is
  \begin{align}
    \Expected{\Length(Y_n)} &=  \Width(\EE^3)
                  \cdot \left[  \tfrac{32}{3 \sqrt{\pi}} \sqrt{n} \cdot (1+o(1)) \right] .
      \label{eqn:Thm-EIV-III4}
  \end{align}
\end{theorem}	
\begin{proof}
  We first prove the relations for the intrinsic volumes,
  \eqref{eqn:Thm-EIV-III1}, \eqref{eqn:Thm-EIV-III2}, and \eqref{eqn:Thm-EIV-III3}.
  For the volume, the extension from $\Sphere$ to $\partial \EE^3$ is straightforward.
  Since linear transformations preserve volume ratios, we have
  $\Vol(Y_n) / \Vol(\EE^3) = \Vol(X_n) / \Vol(\BB^3)$,
  in which we write $X_n = \Linear^{-1} (Y_n)$.
  The expectation of $\Vol(Y_n)$ is therefore $\Vol(\EE^3)$ times the
  expectation of $\Vol(X_n) / \Vol(\BB^3)$.
  The image of the homeoid density under~$\Linear^{-1}$ is
  the uniform measure on $\Sphere$.
  so we get \eqref{eqn:Thm-EIV-III3} from \eqref{eqn:Thm-EIV-I3}.

  For the area, we use Crofton's formula from integral geometry, which says
  that $\Area(Y_n)$ is four times the average area of the orthogonal projection
  of $Y_n$ onto a random plane.
  To state this more formally, let $\Grass{2}{3}$ be the Grassmannian of
  $2$-dimensional planes passing through the origin in~$\RR^3$,
  noting that it is isomorphic to the $2$-dimensional projective plane.
  Letting $\proj{P}{Y_n}$ be the orthogonal projection of the polytope
  onto $P \in \Grass{2}{3}$, Crofton's formula for the area is
  \begin{align}
    \Area(Y_n)  &=  \tfrac{4}{2 \pi} \int_{P \in \Grass{2}{3}}
                                     \Area(\proj{P}{Y_n}) \diff P .
    \label{eqn:Crofton-area}
  \end{align}
  The area of $\proj{P}{Y_n}$ is really the measure of lines orthogonal to $P$
  that intersect $Y_n$.
  Every such line $L \perp P$ corresponds to a line $\Linear^{-1}(L)$
  that intersects $\Linear^{-1} (Y_n)$.
  Similarly, every line $L \perp P$ that intersects $\EE^3$
  corresponds to a line $\Linear^{-1}(L)$ that intersects $\BB^3$.
  Hence,
  \begin{align}
    \frac{\Area(\proj{P}{Y_n})}{\Area(\proj{P}{\EE^3})}
      &=  \frac{\Area(\proj{Q}{X_n})}{\Area(\proj{Q}{\BB^3})} ,
    \label{eqn:ratio-area}
  \end{align}
  in which $X_n = \Linear^{-1}(Y_n)$ and $Q$ is the plane normal to the lines
  $\Linear^{-1}(L)$.
  Fixing $P$, $\Expected{\Area(\proj{P}{Y_n})}$ is
  therefore $\Area(\proj{P}{\EE^3}) / \pi$ times
  $\Expected{\Area(\proj{Q}{X_n})}$.
  The latter is independent of $Q$ and by Crofton's formula
  equal to $\tfrac{1}{4} \, \Expected{\Area(X_n)}$.
  Hence,
  \begin{align}
    \Expected{\Area(Y_n)}  &=  \tfrac{4}{2 \pi} \int_{P \in \Grass{2}{3}}
                               \Expected{\Area(\proj{P}{Y_n})} \diff P            \\
                           &=  \tfrac{4}{2 \pi^2} \int_{P \in \Grass{2}{3}}
                               \Area(\proj{P}{\EE^3})
                               \cdot \Expected{\Area(\proj{Q}{X_n})} \diff P   \\
                           &=  \left[ \tfrac{4}{2 \pi} \int_{P \in \Grass{2}{3}}
                                 \Area(\proj{P}{\EE^3}) \diff P \right]
                               \cdot \tfrac{1}{4 \pi} \Expected{\Area(X_n)} .
    \label{eqn:ellipsoid-area}
  \end{align}
  By Crofton's formula, the first factor in \eqref{eqn:ellipsoid-area}
  is $\Area(\EE^3)$, and by \eqref{eqn:Thm-EIV-I2},
  the second factor is $\tfrac{n-1}{n+1} \tfrac{n-2}{n+2}$,
  which implies the claimed formula for area.
  The proof for the mean width is similar and thus omitted.

  We second prove the relation for the total edge length, \eqref{eqn:Thm-EIV-III4}.
  To that end, we show that for any vector $x \in \Sphere$,
  the length of $\Linear (x)$ is half the length of the projection of $\EE^3$
  onto the line defined by~$x$.
  This implies that the average length of $\Linear (x)$ --- with $x$
  chosen uniformly at random on $\Sphere$ --- is half the mean width of $\EE^3$.
  The directions of the edges of $X_n = \Linear^{-1} (Y_n)$
  are indeed uniformly distributed.
  Therefore, the expected total edge length of $Y_n$ is
  $\tfrac{1}{2} \Width(\EE^3)$ times the expected total edge length of $X_n$,
  and we get \eqref{eqn:Thm-EIV-III4} from \eqref{eqn:Thm-TEL-I}.
  To show the relation between $\Linear(x)$ and the projection of $\EE^3$,
  we assume that the axes of $\EE^3$ are aligned with the coordinate axes of $\RR^3$.
  Equivalently, the linear map that maps $\BB^3$ to $\EE^3$ is represented
  by the diagonal matrix with entries $p, q, r$ along its diagonal.
  The dual ellipsoid, $\DD^3$, is obtained by applying the inverse matrix.
  Equivalently, the points of $\partial \DD^3$ satisfy
  $p^2 y_1^2 + q^2 y_2^2 + r^2 y_3^2 = 1$.
  Let $x = (x_1, x_2, x_3)$ be a unit vector,
  and set $y = (y_1, y_2, y_3)$ with
  $y_i = x_i / (p^2 x_1^2 + q^2 x_2^2 + r^2 x_3^2)^{1/2}$, for $i = 1, 2, 3$.
  By construction, $y$ belongs to $\partial \DD^3$, it is parallel to $x$,
  and its length is
  \begin{align}
    \| y \|  &=  \sqrt{y_1^2+y_2^2+y_3^2}
              =  \tfrac{1}{\sqrt{p^2 x_1^2 + q^2 x_2^2 + r^2 x_3^2}}
              =  \frac{1}{\|\Linear (x)\|}.
  \end{align}
  Since $\DD^3$ is dual to $\EE^3$, this length is one over the half-length
  of the orthogonal projection of $\EE^3$ on the line defined by $x$, as required.
\end{proof}

\medskip
The same arguments work for polytopes generated by a Poisson point process,
thus generalizing Theorems \ref{thm:intrinsic_volumes_III}
and \ref{thm:total_edge_length}
to the case of an ellipsoid with homeoid density.

\bigskip 
\section{Discussion}
\label{sec:9-Discussion}

By focusing on random polytopes that are inscribed into the unit sphere
in $\RR^3$, we find surprisingly elementary proofs for a number
of their stochastic properties.
As an example, we mention that combinatorial arguments together with
Archimedes' Lemma and Crofton's Formula suffice to compute the expected
mean width, area, and volume as functions of the number of vertices.
We mention a number of open questions:
\medskip
\begin{enumerate}
  \item[1.] Is there an elementary explanation for Lemma \ref{lem:shape_vs_size},
    namely that the shape and the size of a random inscribed simplex are independent?
  \item[2.] Are there intuitive geometric reasons for the strikingly simple
    formulas for the intrinsic volumes highlighted in the Introduction?
    Can we generalize the formulas to higher dimensions without losing their appeal?
  \item[3.] What is the meaning of the constant in the expression for the
    total edge length of a random Poisson polytope?
    What is the meaning of the modified Bessel functions appearing in the expressions?
    Can we get a simpler expression for the total edge length in the uniform case?
  \item[4.] Investigate the surprisingly tight correlation between the intrinsic
    volumes of the random inscribed polytopes illustrated in
    Figure \ref{fig:moment-curve}.
  \item[5.] Can we say something about the distributions of the normalized
    intrinsic volume deficiencies?
    The distributions shown in Figure \ref{fig:distributions} seem to be asymmetric,
    growing slower than they decay.
  \item[6.] Prove Conjecture \ref{conj:upper_bounds} about the extremal properties
    of the virtual model.
    Is there a natural optimization criterion based on the total edge length
    that favors inscribed polytopes whose vertices are well spread and whose
    total edge length is on the order of $\sqrt{n}$?
  \item[7.] What is the distribution of vertices of $X_n$ that have degree $k$?
    Asymptotic formulas but no closed-form expressions for Delaunay mosaics
    in $\RR^2$ can be found in \cite{Cal03,Hil05}.
    Are their results also valid for polytopes inscribed in $\Sphere$?
\end{enumerate}

\bigskip
\subsection*{Acknowledgements}
\footnotesize{We thank Dmitry Zaporozhets for directing us to reference \cite{Kab17}
and Anton Mellit for a useful discussion on Bessel functions.}

\bibliographystyle{abbrvurl}
\bibliography{../triangle}
\bigskip

\end{document}